\documentclass[11pt,reqno]{amsart}

\usepackage{amsmath,amsthm,amsfonts,amssymb,latexsym,mathrsfs,color,hyperref}

\newtheorem{theorem}{Theorem}
\newtheorem{corollary}[theorem]{Corollary}
\newtheorem{proposition}[theorem]{Proposition}
\newtheorem{conjecture}[theorem]{Conjecture}

\newcommand{\pk}{{\rm pk\,}}
\newcommand{\lpk}{{\rm pk^{\emph{l}}\,}}

\newcommand{\cpk}{{\rm cpk\,}}
\newcommand{\val}{{\rm val\,}}
\newcommand{\cval}{{\rm cval\,}}

\newcommand{\ms}{\mathfrak{S}_{n+1}}
\newcommand{\msn}{\mathfrak{S}_n}
\newcommand{\msb}{\mathfrak{S}_8}
\newcommand{\mcc}{\mathcal{C}_{n+1}}
\newcommand{\mcn}{{\mathcal C}_n}
\newcommand{\md}{\mathcal{D}_n}
\newcommand{\mdk}{\mathcal{D}_{n-k}}
\newcommand{\fix}{{\rm fix\,}}
\newcommand{\cyc}{{\rm cyc\,}}

\newcommand{\lrf}[1]{\lfloor #1\rfloor}

\title[Permutations by cyclic peaks and valleys]{Enumeration of permutations by number of cyclic occurrence of peaks and valleys}

\author[S.-M.~Ma]{Shi-Mei~Ma}
\address{School of Mathematics and Statistics,
         Northeastern University at Qinhuangdao,
         Hebei 066004, P. R. China}
\email{shimeima@yahoo.com.cn (S.-M. Ma)}
\author[C.-O.~Chow]{Chak-On~Chow}
\address{P.O. Box 91100, Tsimshatsui Post Office, Hong Kong}
\address{Department of Mathematics and Information Technology, Hong Kong Institute of Education, 10 Lo Ping Road, Tai Po, New Territories, Hong Kong}
\email{cchow@alum.mit.edu (C.-O. Chow) }
\thanks{The first author is supported by NSFC (11126217).}

\begin{document}

\maketitle

\begin{abstract}
In this paper, we focus on the enumeration of permutations by number of cyclic occurrence of peaks and valleys.
We find several recurrence relations involving the number of
permutations with a prescribed number of cyclic peaks, cyclic valleys, fixed points and cycles.
Several associated permutation statistics and the corresponding generating functions are also studied. In particular, we establish a connection between cyclic valleys and Pell numbers as well as cyclic peaks and alternating runs.
\bigskip\\
{\sl Keywords:} Cyclic peaks; Cyclic valleys; Fixed points; Cycles; Alternating runs
\smallskip\\
{\sl 2010 Mathematics Subject Classification:} 05A05; 05A15
\end{abstract}
\section{Introduction}
Let $[n]=\{1,2,\ldots,n\}$, and let $\msn$ denote the
the set of permutations of $[n]$.
A permutation $\pi\in\msn$ can be written in {\it one-line notation} as the word $\pi=\pi(1)\pi(2)\cdots\pi(n)$.
Another way of writing the permutation is given by the {\it standard
cycle decomposition}, where each cycle is written with its smallest entry first and the cycles are written in increasing order of their smallest entry. For example, the permutation $\pi=64713258\in\msb$ has the standard
cycle decomposition $(1,6,2,4)(3,7,5)(8)$.

A permutation $\pi=\pi(1)\pi(2)\cdots\pi(n)\in\msn$
is {\it alternating} if $\pi(1)>\pi(2)<\pi(3)>\pi(4)<\cdots$. Similarly, $\pi$ is {\it reverse alternating} if $\pi(1)<\pi(2)>\pi(3)<\pi(4)>\cdots$.
It is well known~\cite{Andre84} that the {\it Euler numbers} $E_n$ defined by
$$\sum_{n=0}^\infty E_n\frac{x^n}{n!}=\tan x+\sec x$$
count alternating permutations in $\msn$.
The first few values of $E_n$ are $1,1,1,2,5,16,61,272,\ldots$.
The bijection $\pi\mapsto \pi^c$ on $\msn$ defined by $\pi^c(i)=n+1-\pi(i)$ shows that $E_n$ is also the number of reverse alternating permutations in $\msn$. The study of the Euler numbers is a topic in combinatorics (see~\cite{Sta10}). For example,
Elizalde and Deutsch~\cite{Elizalde11} recently studied {\it cycle up-down}
permutations. A cycle is said to be up-down if, when written in standard cycle form, say $(b_1,b_2,b_3,\ldots)$, we have $b_1<b_2>b_3<\cdots$. We say that $\pi$ is a cycle up-down permutation if it is a product of up-down cycles.
Elizalde and Deutsch~\cite[Proposition 2.1]{Elizalde11} found that the number of cycle up-down permutations of $[n]$ is $E_{n+1}$ .

There is a wealth of literature on peak statistics
of permutations (see~\cite{Bouchard10,Dilks09,Kitaev03,Ma12,Rieper00,Zeng12,Vella03} for instance).
For example, Kitaev~\cite{Kitaev03} found that there are $2^{n-1}$ permutations of $[n]$ without interior peaks.
Let $\pi=\pi(1)\pi(2)\cdots \pi(n)\in\msn$. An {\it interior peak} (resp. {\it interior valley}) in $\pi$ is an index $i\in\{2,3,\ldots,n-1\}$ such that $\pi(i-1)<\pi(i)>\pi(i+1)$ (resp. $\pi(i-1)>\pi(i)<\pi(i+1)$).
Clearly, interior peaks and interior valleys are equidistributed on $\msn$.
Let $\pk(\pi)$ (resp. $\val(\pi)$) denote the number of
interior peaks (resp. the number of interior valleys) in $\pi$.
An {\it left peak} in $\pi$ is an index $i\in[n-1]$ such that $\pi(i-1)<\pi(i)>\pi(i+1)$, where we take $\pi(0)=0$.
Let $\lpk(\pi)$ denote {\it the number of
left peaks} in $\pi$. For example, the permutation $\pi=64713258\in\msb$ has $\pk(\pi)=2,\val(\pi)=3$ and $\lpk(\pi)=3$.

For $n\geqslant 1$, we define
\begin{equation*}
W_n(q)=\sum_{\pi\in\msn}q^{\pk(\pi)}\quad {\text and}\quad \overline{W}_n(q)=\sum_{\pi\in\msn}q^{\lpk(\pi)}.
\end{equation*}
It is well known that
the polynomials $W_n(q)$ satisfy the recurrence relation
\begin{equation*}
W_{n+1}(q)=(nq-q+2)W_n(q)+2q(1-q)W_n'(q),
\end{equation*}
with initial values $W_1(q)=1$, $W_2(q)=2$ and $W_3(q)=4+2q$,
and the polynomials $\overline{W}_n(q)$ satisfy the recurrence relation
\begin{equation*}
\overline{W}_{n+1}(q)=(nq+1)\overline{W}_{n}(q)+2q(1-q)\overline{W}_{n}'(q),
\end{equation*}
with initial values $\overline{W}_{1}(q)=1$, $\overline{W}_{2}(q)=1+q$ and $\overline{W}_{3}(q)=1+5q$
(see~\cite[\textsf{A008303, A008971}]{Sloane}).
The exponential generating functions of the polynomials $W_n(q)$ and $\overline{W}_n(q)$ are respectively given
as follows (see~\cite{Ma12}):
\begin{equation*}
W(q,z)=\sum_{n\geq 1}W_n(q)\frac{z^n}{n!}= \frac{\sinh(z\sqrt{1-q})}{\sqrt{1-q}\cosh(z\sqrt{1-q})-\sinh(z\sqrt{1-q})},
\end{equation*}
\begin{equation}\label{Wqz-GF}
\overline{W}(q,z)=1+ \sum_{n\geq 1}\overline{W}_n(q)\frac{z^n}{n!}
   =\frac{\sqrt{1-q}}{\sqrt{1-q}\cosh(z\sqrt{1-q})-\sinh(z\sqrt{1-q})}.
\end{equation}

An occurrence of a {\it pattern} $\tau$ in a permutation $\pi$ is defined as a subsequence in $\pi$ whose letters are in the same relative order as those in $\tau$.
For example, the permutation $\pi=64713258\in\msb$ has two occurrences of
the pattern 1--2--3--4, namely the subsequences 1358 and 1258.
In~\cite{Babson00}, Babson and Steingr\'imsson introduced {\it generalized permutation patterns} that allow the requirement that two adjacent letters in a pattern must be adjacent in the permutation. Thus, an occurrence of an interior peak in a permutation is an occurrence of the pattern $132$ or $231$. Similarly, an occurrence of interior valley is an occurrence of the pattern $213$ or $312$.
Recently, Parviainen~\cite{Parviainen06,Parviainen07} explored {\it cyclic occurrence of patterns} over $\msn$ via continued fractions.

In this paper, we focus on the enumeration of permutations by number of cyclic occurrence of peaks and valleys.
The paper is organised as follows.
In Section~\ref{notation}, we collect some notation, definitions and results that will be needed in the rest of the paper. In Section~\ref{Result}, we present several recurrence relations. In Section~\ref{Euler}, we discuss two classes of triangular arrays.
In Section~\ref{GF}, we compute two exponential generating functions of generating functions of permutations by their numbers of cyclic peaks/valleys, cycles and fixed points. In Section~\ref{pell}, we establish a connection between cyclic valleys and the famous Pell numbers. In Section~\ref{Runs}, we establish a connection between cyclic peaks and alternating runs.
\section{Notation, definitions and preliminaries}\label{notation}
In the following discussion we always write a permutation $\pi\in\msn$
in standard cycle decomposition.
Let
\begin{equation*}
\pi=(c_1^1,c_2^1,\ldots c_{i_1}^1)(c_1^2,c_2^2,\ldots c_{i_2}^2)\ldots (c_1^k,c_2^k,\ldots c_{i_k}^k),
\end{equation*}
and let $\sigma$ be a generalised pattern. Following Parviainen~\cite{Parviainen06,Parviainen07},
the pattern $\sigma$ occurs cyclically in $\pi$ if it occurs in the permutation
$$\Psi(\pi)=c_1^1,c_2^1,\ldots c_{i_1}^1,c_1^2,c_2^2,\ldots c_{i_2}^2,\ldots c_1^k,c_2^k,\ldots c_{i_k}^k,$$
with the further restriction that $c_{i_j}^j$ and $c_{1}^{j+1}$ are not adjacent, where $1\leq j\leq k-1$.
For example, 132 does not occur in $(1)(2,4,5)(3)$, but 13--2 occurs exactly twice.
An entry $c_m^j$ in the cycle $(c_1^j,c_2^j,\ldots c_{i_j}^j)$
is called a {\it cyclic peak} (resp. {\it cyclic valley}) of $\pi$ if $c_{m-1}^{j}<c_m^j>c_{m+1}^{j}$ (resp. $c_{m-1}^{j}>c_m^j<c_{m+1}^{j}$), where $2\leqslant m\leqslant i_j-1$ and $1\leqslant  j\leqslant k$.
For example, the permutation $(1,6,2,4)(3,7,5)(8)$ has the cyclic peaks 6 and 7 and
the cyclic valley 2.

Let $\cpk(\pi)$ (resp. $\cval(\pi)$) denote {\it the number of cyclic peaks} (resp. {\it the number of cyclic valleys}) of $\pi$.
The number of {\it fixed points} of $\pi$ is $\fix(\pi)=\#\{1\leqslant i\leqslant
n:\pi(i)=i\}$. A fixed-point-free permutation is called a {\it derangement}.
Let $\md$ denote the set of derangements of $[n]$.
Denote by $\cyc(\pi)$ the number of {\it cycles} of
$\pi$.
For $n\geqslant 1$, we introduce the following generating functions:
$$P_n(q,x,y)=\sum_{\pi\in\msn}q^{\cpk(\pi)}x^{\cyc(\pi)}y^{\fix(\pi)};$$
$$V_n(q,x,y)=\sum_{\pi\in\msn}q^{\cval(\pi)}x^{\cyc(\pi)}y^{\fix(\pi)};$$
$$M_n(q)=\sum_{\pi\in\msn}q^{\cpk(\pi)}=\sum_{k\geqslant 0}M_{n,k}q^k;$$
$$\overline{M}_n(q)=\sum_{\pi\in\msn}q^{\cval(\pi)}=\sum_{k\geqslant 0}\overline{M}_{n,k}q^k;$$
$$D_n(q)=\sum_{\pi\in \md}q^{\cpk(\pi)}=\sum_{k\geqslant 0}D_{n,k}q^k;$$
$$\overline{D}_n(q)=\sum_{\pi\in \md}q^{\cval(\pi)}=\sum_{k\geqslant 0}\overline{D}_{n,k}q^k.$$
Let
$$p(n,t,s,r)=\#\{\pi\in\msn:\cpk(\pi)=t,\cyc(\pi)=s,\fix(\pi)=r\},$$
and let
$$v(n,t,s,r)=\#\{\pi\in\msn:\cval(\pi)=t,\cyc(\pi)=s,\fix(\pi)=r\}.$$

The {\it Stirling numbers of the second kind} ${n \brace k}$ is the number of
partitions of $[n]$ into $k$ blocks. Let $S_n(x)=\sum_{k=1}^n{n \brace k}x^k$.
It is well known (see~\cite[A008277]{Sloane}) that $${n+1 \brace k}={n \brace k-1}+k{n \brace k},$$ which is equivalent to
\begin{equation*}
S_{n+1}(x)=xS_{n}(x)+xS_n'(x).
\end{equation*}
The {\it associated Stirling numbers of second kind} $T(n,k)$ is the number of partitions of $[n]$ into
$k$ blocks of size at least 2. Let $T_n(x)=\sum_{k\geqslant 1}T(n,k)x^k$.
It is well known (see~\cite[A008299]{Sloane}) that $$T(n+1,k)=kT(n,k)+nT(n-1,k-1),$$ which is equivalent to
\begin{equation*}
 T_{n+1}(x)=xT_{n}'(x)+nxT_{n-1}(x).
\end{equation*}
Clearly, $P_n(0,x,1)=S_{n}(x)$ and $P_n(0,x,0)=T_n(x)$. For example,
take a
permutation
$$\pi=(\pi({i_1}),\ldots)(\pi({i_2}),\ldots)\cdots(\pi({i_s}),\ldots)$$
counted by $p(n,0,s,r)$.
Recall that
$$p(n,0,s,r)=\#\{\pi\in \msn:\cpk(\pi)=0, \cyc(\pi)=s, \fix(\pi)=r\}.$$
Erasing the parentheses, we get a partition of $[n]$ with $s$
blocks. Hence $$\sum_{r\geqslant0}p(n,0,s,r)={n \brace s}.$$

We say that a permutation $\pi$ is called a {\it circular permutation} if $\cyc(\pi)=1$.
Denote by $\mcn$ the set of circular permutations of $[n]$. Each $\pi\in\mcc$ can
be written uniquely as a cycle of the form $\pi=(1,a_1,a_2,\ldots,a_n)$.
Let $\varphi(\pi)=b_1b_2\cdots b_n$, where $b_i=a_i-1$ for $1\leqslant i\leqslant n$.
The correspondence $\varphi: \mcc\mapsto \msn$ is clearly a bijection.
Using the bijection $\varphi$, it is clear that the coefficient of $x$ in the polynomial $P_{n+1}(q,x,y)$ (resp. $V_{n+1}(q,x,y)$) is $\overline{W}_{n}(q)$ (resp. $W_n(q)$).
In the next section, we present recurrence relations for
the polynomials $P_{n}(q,x,y)$ and $V_{n}(q,x,y)$.
\section{Recurrence relations}\label{Result}
\begin{theorem}\label{pnstr}
The numbers $p(n,t,s,r)$ satisfy the recurrence relation
\begin{align*}
p(n+1,t,s,r)&=(2t+s-r)p(n,t,s,r)+p(n,t,s-1,r-1)+\\
            &(r+1)p(n,t,s,r+1)+(n+2-2t-s)p(n,t-1,s,r).
\end{align*}
\end{theorem}
\begin{proof}
Let $n$ be a fixed positive integer.
Let
$\sigma_{i}\in \ms$ be the permutation obtained from $\sigma\in
\msn$  by inserting the entry $n+1$ either to the left or to the right of $\sigma(i)$ if
$i\in[n]$ or as a new cycle $(n+1)$ if $i=n+1$. Then
$$\cyc(\sigma_{i})=\begin{cases}
\cyc(\sigma) & \text{if $i\in[n]$},\\ \cyc(\sigma)+1  & \text{if
$i=n+1$;}
\end{cases}$$
and
$$\fix(\sigma_{i})=\begin{cases}
\fix(\sigma)-1 & \text{if $i\in[n]$ and $\sigma(i)=i$},\\
\fix(\sigma) & \text{if $i\in[n]$ and $\sigma(i)\neq i$,}\\
\fix(\sigma)+1 & \text{if $i=n+1$.}
\end{cases}$$

Recall that
\begin{equation*}
p(n+1,t,s,r)=\#\{\sigma_{i}\in \ms:\cpk(\sigma_{i})=t, \cyc(\sigma_{i})=s, \fix(\sigma_{i})=r\}.
\end{equation*}

It is evident that $\cpk(\sigma_{i})=\cpk(\sigma)$ or $\cpk(\sigma_{i})=\cpk(\sigma)+1$.
There are four cases to consider.
\begin{enumerate}
  \item [(a)]If $\sigma$ has $t$ cyclic peaks, $s$ cycles and $r$ fixed points, then we can
put the entry $n+1$ on either side of some cyclic peak or at the end of a cycle of length greater than or equal to 2. This means we have $2t+s-r$ choices for the positions of $n+1$.
As we have $p(n,s,t,r)$ choices for $\sigma$,
the first term of the recurrence relation is obtained.
  \item [(b)]If $\sigma$ has $t$ cyclic peaks, $s-1$ cycles and $r-1$ fixed points, then we can insert
the entry $n+1$ at the end of $\sigma$ to form a new cycle $(n+1)$.
As we have $p(n,t,s-1,r-1)$ choices for $\sigma$, the second term of
the recurrence relation is obtained.
  \item [(c)]If $\sigma$ has $t$ cyclic peaks, $s$ cycles and $r+1$ fixed points, then we can insert
the entry $n+1$ to the right of a fixed points. This means we have $r+1$ choices for the positions of $n+1$.
This case
applies to the third term in the recurrence relation.
  \item [(d)]If $\sigma$ has $t-1$ cyclic peaks, $s$ cycles and $r$ fixed points, then we can insert
the entry $n+1$ in one of the $n-2(t-1)-s=n+2-2t-s$ middle positions and the last term of
the recurrence relation is obtained.
\end{enumerate}
This completes the proof.
\end{proof}

We can express Theorem~\ref{pnstr} in terms of differential operators.

\begin{corollary}\label{Pnqxy-recu}
For $n\geqslant 1$, we have
\begin{align*}
P_{n+1}(q,x,y)&=(nq+xy) P_n(q,x,y)+2q(1-q)\frac{\partial P_n(q,x,y)}{\partial q}+x(1-q)\frac{\partial P_n(q,x,y)}{\partial x}\\
&\quad+(1-y)\frac{\partial P_n(q,x,y)}{\partial y}.
\end{align*}
In particular,
\begin{equation}\label{Pnqx1}
P_{n+1}(q,x,1)=(nq+x) P_n(q,x,1)+2q(1-q)\frac{\partial P_n(q,x,1)}{\partial q}+x(1-q)\frac{\partial P_n(q,x,1)}{\partial x}.
\end{equation}
\end{corollary}
By Corollary~\ref{Pnqxy-recu}, we can easily compute the first few polynomials $P_n(q,x,y)$:
\begin{align*}
  P_1(q,x,y)& =xy, \\
  P_2(q,x,y)& =x+x^2y^2, \\
  P_3(q,x,y)& =(1+q)x+3x^2y+x^3y^3,\\
  P_4(q,x,y)& =(1+5q)x+(3+4y+4qy)x^2+6x^3y^2+x^4y^4.
\end{align*}

Recall that $M_n(q)=P_n(q,1,1)$. For $1\leqslant n\leqslant 7$, using~\eqref{Pnqx1}, the coefficients of $M_n(q)$
can be arranged as follows with $M_{n,k}$ in row $n$ and column $k$:
$$\begin{array}{cccccccc}
  1 &  &  &  & & &&\\
  2 &  &  &  & & &&\\
  5 & 1&  &  & & &&\\
  15 & 9 &  &  & &  &&\\
  52 & 63 & 5 &  &  & &&\\
  203 & 416 & 101&  &  &   &&\\
  877 & 2741& 1361 & 61 & &  & &
\end{array}$$

Let $D_{n,k}(q,x)$ be the coefficient of $y^k$ in $P_n(q,x,y)$.
Clearly, the polynomial $P_n(q,x,0)$ is the corresponding
enumerative polynomial on $D_n$. Note that
\begin{eqnarray*}
D_{n,k}(q,x)&=&\sum_{\stackrel{\pi\in\msn}{\fix(\pi)=k}}q^{\cpk(\pi)}x^{\cyc(\pi)}\\
             &=&\binom{n}{k}x^{k}\sum_{\sigma\in
            \mdk}q^{\cpk(\sigma)}x^{\cyc(\sigma)}\\
             &=&\binom{n}{k}x^{k}P_{n-k}(q,x,0).
\end{eqnarray*}

Thus $$\left(\frac{\partial P_n(q,x,y)}{\partial y}\right)_{y=0}=D_{n,1}(q,x)=nxP_{n-1}(q,x,0).$$
Therefore, we get the following result.

\begin{proposition}
For $n\geqslant 2,$ the polynomials $P_n(q,x,0)$ satisfy the following recurrence relation
\begin{equation}\label{Pnqx0}
P_{n+1}(q,x,0)=nqP_n(q,x,0)+2q(1-q)\frac{\partial P_n(q,x,0)}{\partial q}+x(1-q)\frac{\partial P_n(q,x,0)}{\partial x}+nxP_{n-1}(q,x,0),
\end{equation}
with initial values $P_0(q,x,0)=1,P_1(q,x,0)=0,P_2(q,x,0)=x$ and $P_3(q,x,0)=(1+q)x$.
\end{proposition}
Recall that $D_n(q)=P_n(q,1,0)$. For $1\leqslant n\leqslant 7$, using~\eqref{Pnqx0}, the coefficients of $D_n(q)$
can be arranged as follows with $D_{n,k}$ in row $n$ and column $k$:
$$\begin{array}{cccccccc}
  0 &  &  &  & & &&\\
  1 &  &  &  & & &&\\
  1 & 1&  &  & & &&\\
  4 & 5 &  &  & &  &&\\
  11 & 28 & 5 &  &  & &&\\
  41 & 153 & 71&  &  &   &&\\
  162 & 872& 759 & 61 & &  & &
\end{array}$$

In order to provide a recurrence relation for the numbers $v(n,t,s,r)$, we first define an operation. Assume that $\pi$ is a permutation in $\msn$ with $k$ cycles $C_1,C_2,\ldots,C_k$, where $C_j=(c_1^j,c_2^j,\ldots, c_{i_j}^j)$ and $1\leqslant j\leqslant k$.
Let $\phi: \msn\rightarrow \msn$ be defined as follows:
\begin{itemize}
  \item $\phi(C_1,C_2,\ldots,C_k)=(\phi(C_1),\phi(C_2),\ldots,\phi(C_k))$.
  \item  For every cycle $C_j$, we have
$$\phi(c_1^j,c_2^j,\ldots,c_{i_j}^j)=(\phi(c_1^j),\phi(c_2^j),\ldots, \phi(c_{i_j}^j)).$$
  \item Let $\{a_1,a_2,\ldots,a_{i_j}\}$ be the set of entries of the cycle $C_j=(c_1^j,c_2^j,\ldots, c_{i_j}^j)$, and assume that $a_1<a_2<\cdots <a_{i_j}$. Then $\phi(a_m)=a_{i_j+1-m}$, where $1\leqslant m\leqslant i_j$.
\end{itemize}
For example, $\phi((1,6,2,4)(3,7,5)(8))=(6,1,4,2)(7,3,5)(8)$.
Following~\cite{Johnson00}, we
call this operation a {\it switching}.
Clearly, if $\pi$ has $t$ cyclic valleys,
then $\phi(\pi)$ has $t$ cyclic peaks, and vice versa.

\begin{theorem}\label{vntsr}
The numbers $v(n,t,s,r)$ satisfy the recurrence relation
\begin{align*}
v(n+1,t,s,r)&=(2t+2s-2r)v(n,t,s,r)+v(n,t,s-1,r-1)\\
            &\quad+(r+1)v(n,t,s,r+1)+(n+2-2t-2s+r)v(n,t-1,s,r).
\end{align*}
\end{theorem}
\begin{proof}
Let $n$ be a fixed positive integer.
Let
$\sigma_{i}\in \ms$ be the permutation obtained from $\sigma\in
\msn$  by inserting the entry $n+1$ either to the left or to the right of $\sigma(i)$ if
$i\in[n]$ or as a new cycle $(n+1)$ if $i=n+1$.
Recall that
\begin{equation*}
v(n+1,t,s,r)=\#\{\sigma_{i}\in \ms:\cval(\sigma_{i})=t, \cyc(\sigma_{i})=s, \fix(\sigma_{i})=r\}.
\end{equation*}

It is evident that $\cval(\sigma_{i})=\cval(\sigma)$ or $\cval(\sigma_{i})={\cval(\sigma)}+1$.
There are four cases to consider.
\begin{enumerate}
\item [(a)]If $\sigma$ has $t$ cyclic valleys, $s$ cycles and $r$ fixed points, then $\phi(\sigma)$ is a permutation with $t$ cyclic peaks, $s$ cycles and $r$ fixed points. Consider the permutation $\phi(\sigma)$, we can appending $n+1$ either at the beginning or at the end of a cycle of length greater than or equal to 2. We can also put the entry $n+1$ on either side of some cyclic peak of $\phi(\sigma)$. This means we have $2t+2s-2r$ choices for the positions of $n+1$. As we have $v(n,s,t,r)$ choices for $\sigma$, the first term of the the recurrence relation is explained.
\item [(b)]If $\sigma$ has $t$ cyclic valleys, $s-1$ cycles and $r-1$ fixed points, then we can insert the entry $n+1$ at the end of $\sigma$ to form a new cycle $(n+1)$. As we have $v(n,t,s-1,r-1)$ choices for $\sigma$, the second term of the recurrence relation is obtained.
\item [(c)]If $\sigma$ has $t$ cyclic valleys, $s$ cycles and $r+1$ fixed points, then we can insert the entry $n+1$ to the right of a fixed points. This means we have $r+1$ choices for the positions of $n+1$. This case applies to the third term of the recurrence relation.
\item [(d)]If $\sigma$ has $t-1$ cyclic valleys, $s$ cycles and $r$ fixed points, then we can insert the entry $n+1$ in one of the $n-s-(s-r)-2(t-1)=n+2-2t-2s+r$ middle positions. As we have $v(n,t-1,s,r)$ choices for $\sigma$, the last term of the recurrence relation is obtained.
\end{enumerate}
This completes the proof.
\end{proof}

We can express Theorem~\ref{vntsr} in terms of differential operators.

\begin{corollary}\label{Vnqxy-recu}
For $n\geqslant 1$, we have
\begin{align*}
V_{n+1}(q,x,y)&=(nq+xy)V_n+2q(1-q)\frac{\partial V_n(q,x,y)}{\partial q}+2x(1-q)\frac{\partial V_n(q,x,y)}{\partial x}\\
&\quad+(1-2y+qy)\frac{\partial V_n(q,x,y)}{\partial y}.
\end{align*}
\end{corollary}
By Corollary~\ref{Vnqxy-recu}, we can easily compute the first few polynomials $V_n(q,x,y)$:
\begin{align*}
  V_1(q,x,y)& =xy, \\
  V_2(q,x,y)& =x+x^2y^2, \\
  V_3(q,x,y)& =2x+3x^2y+x^3y^3,\\
  V_4(q,x,y)& =(4+2q)x+(3+8y)x^2+6x^3y^2+x^4y^4,\\
  V_5(q,x,y)& =(8+16q)x+(20+20y+10qy)x^2+(15y+20y^2)x^3+10x^4y^3+x^5y^5.
\end{align*}

Recall that $\overline{M}_n(q)=V_n(q,1,1)$. For $1\leqslant n\leqslant 7$, the coefficients of $\overline{M}_n(q)$
can be arranged as follows with $\overline{M}_{n,k}$ in row $n$ and column $k$:
$$
\begin{array}{cccccccc}
  1 &  &  &  & & &&\\
  2 &  &  &  & & &&\\
  6 & &  &  & & &&\\
  22 & 2 &  &  & &  &&\\
  94 & 26 &  &  &  & &&\\
  460 & 244 & 16&  &  &   &&\\
   2532&2124 &384 &  & &  & &
\end{array}
$$

Let $\overline{D}_{n,k}(q,x)$ be the coefficient of $y^k$ in $V_n(q,x,y)$. Clearly,
$\overline{D}_{n,k}(q,x)=\binom{n}{k}x^{k}V_{n-k}(q,x,0)$.
Then $$\left(\frac{\partial V_n(q,x,y)}{\partial y}\right)_{y=0}=\overline{D}_{n,1}(q,x)=nxV_{n-1}(q,x,0).$$
Hence we get the following result.

\begin{proposition}
For $n\geqslant 2$, the polynomials $V_n(q,x,0)$ satisfy the following recurrence relation
\begin{equation}\label{Vnqx0}
V_{n+1}(q,x,0)=nqV_n(q,x,0)+2q(1-q)\frac{\partial V_n(q,x,0)}{\partial q}+2x(1-q)\frac{\partial V_n(q,x,0)}{\partial x}+nxV_{n-1}(q,x,0),
\end{equation}
with initial values $V_0(q,x,0)=1,V_1(q,x,0)=0,V_2(q,x,0)=x$ and $V_3(q,x,0)=2x$.
\end{proposition}

Recall that $\overline{D}_n(q)=V_n(q,1,0)$. For $1\leqslant n\leqslant 7$, using~\eqref{Vnqx0}, the coefficients of $\overline{D}_n(q)$
can be arranged as follows with $\overline{D}_{n,k}$ in row $n$ and column $k$:
$$\begin{array}{cccccccc}
  0 &  &  &  & & &&\\
  1 &  &  &  & & &&\\
  2 & &  &  & & &&\\
  7 & 2 &  &  & &  &&\\
  28 & 16 &  &  &  & &&\\
  131 & 118 & 16&  &  &   &&\\
  690& 892 &272&  & &  & &
\end{array}$$

\section{On combinations of polynomials and Euler numbers}\label{Euler}
Recall that $M_n(q)=P_n(q,1,1)$, $\overline{M}_n(q)=V_n(q,1,1)$, $D_n(q)=P_n(q,1,0)$ and $\overline{D}_n(q)=V_n(q,1,0)$.
It is easy to verify that $\deg M_n(q)=\deg D_n(q)=\lrf{\frac{n-1}{2}}$ and $\deg \overline{M}_n(q)=\deg \overline{D}_n(q)=\lrf{\frac{n}{2}}-1$ for $n\geqslant 2$.

We define
$$
R_{n}(q)=M_n(q^2)+q\overline{M}_n(q^2) \quad\textrm{for $n\geqslant 3$}.
$$
Let $R_{n}(q)=\sum_{k\geqslant 0}R_{n,k}q^k$. Then for $n\geqslant 3$, we have
\begin{equation*}\label{Trirelation}
R_{n,k}=\begin{cases}
\overline{M}_{n,\frac{k-1}{2}}& \text{if $k$ is odd},\\
M_{n,\frac{k}{2}} & \text{if $k$ is even}.
\end{cases}
\end{equation*}

The $n$th {\it Bell number} $B_n$ counts the number of partitions of $[n]$ into non-empty blocks (see~\cite[A000110]{Sloane} for details).
Recall that $M_{n,0}=\#\{\pi\in\msn:\cpk(\pi)=0\}$. Take a
permutation
$$\pi=(\pi({i_1}),\ldots)(\pi({i_2}),\ldots)\cdots(\pi({i_j}),\ldots)$$
counted by $M_{n,0}$. Erasing the parentheses, we get a partition of $[n]$ with $j$
blocks. Hence
\begin{equation}\label{eqn:02}
R_{n,0}=M_{n,0}=B_n.
\end{equation}
Set $R_1(x)=1$ and $R_2(x)=2+x$. For $1\leqslant n\leqslant 6$, the coefficients of $R_n(q)$
can be arranged as follows with $R_{n,k}$ in row $n$ and column $k$:
$$\begin{array}{cccccccc}
  1 &  &  &  & & &&\\
  2 & 1 &  &  & & &&\\
  5 & 6 & 1 &  & & &&\\
  15 & 22 & 9 & 2 & &  &&\\
  52 & 94 & 63 & 26 & 5 & &&\\
  203 & 460 & 416& 244 & 101 & 16  &&\\
\end{array}$$

We define
\begin{equation*}
I_{n}(q)=D_n(q^2)+q\overline{D}_n(q^2) \quad\textrm{for $n\geqslant 2$}.
\end{equation*}
Let $I_{n}(q)=\sum_{k\geqslant 0}I_{n,k}q^k$. Then for $n\geqslant 2$, we have
\begin{equation*}\label{Trirelation}
I_{n,k}=\begin{cases}
\overline{D}_{n,\frac{k-1}{2}}& \text{if $k$ is odd},\\
D_{n,\frac{k}{2}} & \text{if $k$ is even}.
\end{cases}
\end{equation*}

Recall that $D_{n,0}=\#\{\pi\in \md:\cpk(\pi)=0\}$. Take a
permutation
$$\pi=(\pi({p_1}),\ldots)(\pi({p_2}),\ldots)\cdots(\pi({p_k}),\ldots)$$
counted by $D_{n,0}$. When $n\geqslant 2$, erasing the parentheses, we get a partition of $[n]$ into $k$ blocks of size at least 2. Therefore, $D_{n,0}=\sum_{k\geqslant 1}T(n,k)$ for $n\geqslant 2$, where $T(n,k)$ is the associated Stirling numbers of second kind.
Set $I_1(q)=1$. For $1\leqslant n\leqslant 6$, the coefficients of $I_n(q)$
can be arranged as follows with $I_{n,k}$ in row $n$ and column $k$:
$$\begin{array}{cccccccc}
  1 &  &  &  & & &&\\
  1 & 1 &  &  & & &&\\
  1 & 2 & 1 &  & & &&\\
  4 & 7 & 5 & 2 & &  &&\\
  11 & 28 & 28 & 16 & 5& &&\\
  41 & 131 & 153 &118  & 71&16 &&\\
\end{array}$$

Note that
$$M_{2k+1,k}=D_{2k+1,k}=\#\{\pi\in S_{2k+1}:\cpk(\pi)=k\}$$
and
$$\overline{M}_{2k+2,k}=\overline{D}_{2k+2,k}=\#\{\pi\in S_{2k+2}:\cval(\pi)=k\}.$$
Then from~\cite[Proposition 2.1]{Elizalde11}, we get the following result.
\begin{proposition}
For $n\geqslant 0$, we have
$R_{n+1,n}=I_{n+1,n}=E_{n}.$
\end{proposition}

Let $a_0,a_1,a_2\ldots,a_n$ be a sequence of nonnegative real numbers.
The sequence is said to be {\it log-concave} if $a_{i-1}a_{i+1}\leqslant {a^2_i}$ for $i=1,2,\ldots,n-1$.
We end this section by proposing the following.
\begin{conjecture}
For $n\geqslant 1$, the sequences of coefficients of the polynomials $R_n(x)$ and $I_n(x)$ are both log-concave.
\end{conjecture}
\section{Generating functions}\label{GF}
We present in the present section the exponential generating functions of $P_n(q,x,y)$ and $V_n(q,x,y)$. Towards this end, we define
$$
P(q,x,y,z):=\sum_{n\geqslant 0}P_n(q,x,y)\frac{z^n}{n!}\qquad\textrm{and}\qquad V(q,x,y,z):=\sum_{n\geqslant 0}V_n(q,x,y)\frac{z^n}{n!}.
$$
Instead of computing $P(q,x,y,z)$ and $V(q,x,y,z)$ directly, we first consider $P(q,x,0,z)$ and $V(q,x,0,z)$ since their recurrence relations involve less partial derivatives.

\begin{theorem}\label{mthm:01}
The generating functions $P=P(q,x,0,z)$ and $V=V(q,x,0,z)$ satisfy the following partial differential equations $($PDEs$)$
\begin{equation*}
\begin{split}
(1-qz)\frac{\partial P}{\partial z}+2q(q-1)\frac{\partial P}{\partial q}+x(q-1)\frac{\partial P}{\partial x}&=xzP,\\
(1-qz)\frac{\partial V}{\partial z}+2q(q-1)\frac{\partial V}{\partial q}+2x(q-1)\frac{\partial V}{\partial x}&=xzV,
\end{split}
\end{equation*}
and whose solutions are
\begin{equation*}
\begin{split}
P(q,x,0,z)&=e^{-xz}\biggl[\biggl(\frac{\sqrt{q}-1}{\sqrt{q}+1}\biggr)\biggl(\frac{\sqrt{q}+\cos(z\sqrt{q-1})+\sqrt{q-1}\sin(z\sqrt{q-1})}{\sqrt{q}-\cos(z\sqrt{q-1})-\sqrt{q-1}\sin(z\sqrt{q-1})}\biggr)\biggr]^{\frac{x}{2\sqrt{q}}},\\
V(q,x,0,z)&=e^{-xz/q}\biggl(\frac{\sqrt{q-1}}{\sqrt{q-1}\cos(z\sqrt{q-1})-\sin(z\sqrt{q-1})}\biggr)^{\frac{x}{q}},
\end{split}
\end{equation*}
respectively.
\end{theorem}
\begin{proof}
We first prove the assertions for $V=V(q,x,0,z)$.
Multiplying \eqref{Vnqx0} by $z^n/n!$, followed by summing over $n\geqslant 1$, and using the fact that $V_1(q,x,0)=0$, we get
\begin{equation*}
\begin{split}
\sum_{n\geqslant 1}V_{n+1}(q,x,0)\frac{z^n}{n!}&=\sum_{n\geqslant 1}\biggl[nqV_n(q,x,0)+2q(1-q)\frac{\partial V_n(q,x,0)}{\partial q}+2x(1-q)\frac{\partial V_n(q,x,0)}{\partial x}\\
&\quad+nxV_{n-1}(q,x,0)\biggr]\frac{z^n}{n!}
\end{split}
\end{equation*}
whence the PDE for $V$:
\begin{equation}\label{eqn:03}
(1-qz)\frac{\partial V}{\partial z}+2q(q-1)\frac{\partial V}{\partial q}+2x(q-1)\frac{\partial V}{\partial x}=xzV.
\end{equation}
In view of the so-called $\beta$-extension \cite{fz88} (see also \cite[\S7]{b00} for an instance of it), we may assume that $V(q,x,0,z)=v(q,z)^x$ for some function $v$. Then
$$
\frac{\partial V}{\partial z}=xv^{x-1}\frac{\partial v}{\partial z},\quad\frac{\partial V}{\partial q}=xv^{x-1}\frac{\partial v}{\partial q},\quad
\frac{\partial V}{\partial z}=v^x\ln v
$$
so that \eqref{eqn:03} becomes
\begin{equation}\label{eqn:04}
(1-qz)\frac{\partial v}{\partial z}+2q(q-1)\frac{\partial v}{\partial q}=zv-2(q-1)v\ln v.
\end{equation}
Next, we let $w=\ln v$. Then
$$
\frac{1}{v}\frac{\partial v}{\partial z}=\frac{\partial w}{\partial z},\qquad\frac{1}{v}\frac{\partial v}{\partial q}=\frac{\partial w}{\partial q}
$$
and \eqref{eqn:04} becomes the following linear PDE:
\begin{equation}\label{eqn:05}
(1-qz)\frac{\partial w}{\partial z}+2q(q-1)\frac{\partial w}{\partial q}=z-2(q-1)w
\end{equation}
with auxiliary system
$$
\frac{dz}{1-qz}=\frac{dq}{2q(q-1)}=\frac{dw}{z-2(q-1)w}.
$$
Solving the ordinary differential equation arising from the first equality, namely,
$$
\frac{dz}{dq}+\frac{z}{2(q-1)}=\frac{1}{2q(q-1)},
$$
we obtain the characteristic defined by
$$
c_1=\tan^{-1}\sqrt{q-1}-z\sqrt{q-1},
$$
where $c_1$ is an arbitrary constant. On this characteristic, \eqref{eqn:05} becomes the following linear ODE:
$$
2q(q-1)\frac{dw}{dq}=\frac{\tan^{-1}\sqrt{q-1}-c_1}{\sqrt{q-1}}-2(q-1)w
$$
whose solution is
\begin{equation*}
\begin{split}
w(q,z)&=-\frac{z}{q}+\frac{1}{2q}\ln\biggl(\frac{q-1}{q}\biggr)+\frac{c_2}{q}\\
&=-\frac{z}{q}+\frac{1}{2q}\ln\biggl(\frac{q-1}{q}\biggr)+\frac{f(\tan^{-1}\sqrt{q-1}-z\sqrt{q-1})}{q}
\end{split}
\end{equation*}
where $c_2=f(c_1)$ and $f$ is a function to be determined.

Since $V(q,x,0,0)=v(q,0)^x=1$ for all $x$, it follows that $v(q,0)=1$ so that
$$
f(\tan^{-1}\sqrt{q-1})=\frac{1}{2}\ln\biggl(\frac{q}{q-1}\biggr).
$$
Letting $u=\tan^{-1}\sqrt{q-1}$, we then have $q=1+\tan^2u=\sec^2u$ and hence
$$
f(u)=\frac{1}{2}\ln\biggl(\frac{\sec^2u}{\sec^2u-1}\biggr)=-\ln\sin u.
$$
Since
\begin{equation*}
\begin{split}
f(\tan^{-1}\sqrt{q-1}-z\sqrt{q-1})&=-\ln\sin(z\sqrt{q-1}-z\sqrt{q-1})\\
&=\ln\biggl(\frac{\sqrt q}{\sqrt{q-1}\cos(z\sqrt{q-1})-\sin(z\sqrt{q-1})}\biggr),
\end{split}
\end{equation*}
we obtain that
$$
V(q,x,0,z)=v(q,z)^x=e^{wx}=e^{-xz/q}\biggl(\frac{\sqrt{q-1}}{\sqrt{q-1}\cos(z\sqrt{q-1})-\sin(z\sqrt{q-1})}\biggr)^{\frac{x}{q}}.
$$
By imitating the above calculations, one can obtain the corresponding assertions for $P=P(q,x,0,z)$, whose details are left to the interested readers.
\end{proof}

\begin{corollary}\label{pqxyz-vqxyz}
The exponential generating functions $P(q,x,y,z)$ and $V(q,x,y,z)$ have the following explicit expressions$:$
\begin{equation*}
\begin{split}
P(q,x,y,z)&=e^{xz(y-1)}\biggl[\biggl(\frac{\sqrt{q}-1}{\sqrt{q}+1}\biggr)
\biggl(\frac{\sqrt{q}+\cos(z\sqrt{q-1})+\sqrt{q-1}\sin(z\sqrt{q-1})}
{\sqrt{q}-\cos(z\sqrt{q-1})-\sqrt{q-1}\sin(z\sqrt{q-1})}\biggr)\biggr]^{\frac{x}{2\sqrt{q}}},\\
V(q,x,y,z)&=e^{xz(y-1/q)}\biggl(\frac{\sqrt{q-1}}{\sqrt{q-1}\cos(z\sqrt{q-1})-\sin(z\sqrt{q-1})}\biggr)^{\frac{x}{q}},
\end{split}
\end{equation*}
respectively.
\end{corollary}
\begin{proof}
Let $\sigma\in\mathfrak{S}_n$ have $k$ fixed points. There are $\binom{n}{k}$ choices of fixed points. If $\tau$ is the partial permutation obtained by deleting all the fixed points of $\sigma$, then $\cyc(\sigma)=\cyc(\tau)+k$ and $\cval(\sigma)=\cval(\tau)$ since only $\ell$-cycles ($\ell\geqslant 3$) of $\sigma$ contribute to $\cval(\sigma)$. Thus, $V_n(q,x,y)=\sum_{k=0}^n\binom{n}{k}V_{n-k}(q,x,0)(xy)^k$ follows. Hence,
\begin{equation*}
\begin{split}
\sum_{n\geqslant 0}V_n(q,x,y)\frac{z^n}{n!}&=\sum_{n\geqslant 0}\sum_{k=0}^n\binom{n}{k}V_{n-k}(q,x,0)(xy)^k\frac{z^n}{n!}\\
&=\sum_{k\geqslant 0}\frac{(xyz)^k}{k!}\sum_{n\geqslant k}V_{n-k}(q,x,0)\frac{z^{n-k}}{(n-k)!}\\
&=e^{xyz}e^{-xz/q}\biggl(\frac{\sqrt{q-1}}{\sqrt{q-1}\cos(z\sqrt{q-1})-\sin(z\sqrt{q-1})}\biggr)^{\frac{x}{q}}.
\end{split}
\end{equation*}
thus proving the assertion for $V(q,x,y,z)$. The corresponding assertion for $P(q,x,y,z)$ follows from similar consideration.
\end{proof}

In Section~\ref{Euler}, we combinatorially prove
that the $n$th Bell number $B_n$ is the constant term of $M_n(0)$, i.e.,
$$
M_n(0)=M_{n,0}=B_n.
$$
We now present a generating function proof of this result.
Since $M_n(q)=P_n(q,1,1)$, we have
\begin{equation*}
\begin{split}
\sum_{n\geqslant 0}M_n(0)\frac{z^n}{n!}&=\lim_{q\to 0}\sum_{n\geqslant 0}P_n(q,1,1)\frac{z^n}{n!}\\
&=\lim_{q\to 0}\biggl[\biggl(\frac{\sqrt{q}-1}{\sqrt{q}+1}\biggr)
\biggl(\frac{\sqrt{q}+\cos(z\sqrt{q-1})+\sqrt{q-1}\sin(z\sqrt{q-1})}{\sqrt{q}-\cos(z\sqrt{q-1})-\sqrt{q-1}\sin(z\sqrt{q-1})}\biggr)\biggr]^{\frac{1}{2\sqrt{q}}}.
\end{split}
\end{equation*}
Denote the limit on the right by $L$. It is easy to see that $L$ is of the indeterminate form $1^\infty$.
So, by l'H\^opital's rule, we have
\begin{equation*}
\begin{split}
\ln L&=\lim_{q\to 0}\frac{1}{2\sqrt{q}}\ln\biggl(\frac{\sqrt{q}-1}{\sqrt{q}+1}\biggr)
\biggl(\frac{\sqrt{q}+\cos(z\sqrt{q-1})+\sqrt{q-1}\sin(z\sqrt{q-1})}{\sqrt{q}-\cos(z\sqrt{q-1})-\sqrt{q-1}\sin(z\sqrt{q-1})}\biggr)\\
&=\cosh z+\sinh z-1=e^z-1.
\end{split}
\end{equation*}
Consequently,
$$
\sum_{n\geqslant 0}M_n(0)\frac{z^n}{n!}=e^{e^z-1},
$$
the right side being the exponential generating function of $B_n$, thus proving $M_n(0)=B_n$.

Let $i=\sqrt{-1}$. Note that $\cosh (x)=\cos(ix)$ and $\sinh(x)=-i\sin(ix)$. Combining~\eqref{Wqz-GF}
and Corollary~\ref{pqxyz-vqxyz}, we get the following result.
\begin{theorem}
We have
$$V(q,x,y,z)=e^{xz(y-1/q)}{\overline{W}(q,z)}^{\frac{x}{q}}.$$
\end{theorem}
\section{A Pell number identity}\label{pell}
The Pell numbers $P_n$ are defined by the recurrence relation
$$
P_n=2P_{n-1}+P_{n-2},\qquad n=2,3,\ldots,
$$
with initial values $P_0=0$ and $P_1=1$. Other known facts about the Pell numbers include the Binet formula: for $n=0,1,2,\ldots$,
$$
P_n=\frac{(1+\sqrt{2})^n-(1-\sqrt{2})^n}{2\sqrt{2}},
$$
and the exponential generating function:
$$
\sum_{n=0}^\infty P_n\frac{y^n}{n!}=\frac{e^{(1+\sqrt{2})y}-e^{(1-\sqrt{2})y}}{2\sqrt{2}}.
$$
The Pell numbers can be expressed as $(-1)$-evaluation of $V_n(q,1,0)$, as the next theorem shows.
\begin{theorem}
We have for $n\geqslant 1,$
\begin{equation}\label{eqn:01}
\sum_{\sigma\in \md}(-1)^{\cval(\sigma)}=P_{n-1}.
\end{equation}
\end{theorem}
\begin{proof}
By setting $x=1$ and $q=-1$ in $V(q,x,0,z)$, we see that
\begin{equation*}
\begin{split}
V(-1,1,0,z)&=e^z\Bigl(\sqrt{-2}\cos(z\sqrt{-2})-\sin(z\sqrt{-2})\Bigr)(-2)^{-1/2}\\
&=\frac{e^z[(\sqrt{2}+1)e^{-\sqrt{2}z}+(\sqrt{2}-1)e^{\sqrt{2}z}]}{2\sqrt{2}}\\
&=1+\sum_{n\geqslant 1}P_{n-1}\frac{z^n}{n!},
\end{split}
\end{equation*}
since
\begin{equation*}
\begin{split}
\sum_{n\geqslant 1}P_{n-1}\frac{z^n}{n!}&=\int_0^z\sum_{n\geqslant 0}P_n\frac{s^n}{n!}\,ds\\
&=\frac{1}{2\sqrt{2}}\int_0^z(e^{(1+\sqrt{2})s}-e^{(1-\sqrt{2})s})\,ds\\
&=\frac{e^z[(\sqrt{2}-1)e^{\sqrt{2}z}+(\sqrt{2}+1)e^{-\sqrt{2}z}]}{2\sqrt{2}}-1.
\end{split}
\end{equation*}
Equating the coefficients of $z^n$ on both sides, \eqref{eqn:01} follows.
\end{proof}
\section{Relationship to alternating runs}\label{Runs}
Let $\pi=\pi(1)\pi(2)\cdots \pi(n)\in\msn$. We say that $\pi$ changes
direction at position $i$ if either $\pi({i-1})<\pi(i)>\pi(i+1)$, or
$\pi(i-1)>\pi(i)<\pi(i+1)$, where $i\in\{2,3,\ldots,n-1\}$. We say that $\pi$ has $k$ {\it alternating
runs} if there are $k-1$ indices $i$ such that $\pi$ changes
direction at these positions. Let $R(n,k)$ denote the number of
permutations in $\msn$ with $k$ alternating runs.
Andr\'e~\cite{Andre84} found that the numbers $R(n,k)$ satisfy the following recurrence relation
\begin{equation}\label{rnk-recurrence}
R(n,k)=kR(n-1,k)+2R(n-1,k-1)+(n-k)R(n-1,k-2)
\end{equation}
for $n,k\geqslant 1$, where $R(1,0)=1$ and $R(1,k)=0$ for $k\geqslant 1$.

For $n\geqslant 1$, we define
$R_n(q)=\sum_{k=1}^{n-1}R(n,k)q^k$.
Then by~\eqref{rnk-recurrence}, we obtain
\begin{equation}\label{Rnx-recurrence}
R_{n+2}(q)=q(nq+2)R_{n+1}(q)+q\left(1-q^2\right)R_{n+1}'(q),
\end{equation}
with initial value $R_1(q)=1$. The first few terms of $R_n(q)$'s are given as follows:
\begin{align*}
  R_2(q)& =2q, \\
  R_3(q)& =2q+4q^2, \\
  R_4(q)& =2q+12q^2+10q^3,\\
  R_5(q)& =2q+28q^2+58q^3+32q^4.
\end{align*}

There is a large literature devoted to the numbers $R(n,k)$ (see~\cite[\textsf{A059427}]{Sloane}).
The reader is referred to~\cite{CW08,Ma122,Sta08} for recent progress on this subject.
In a series of papers~\cite{Carlitz78,Carlitz80,Carlitz81}, Carlitz studied
generating functions for the numbers $R(n,k)$. In particular,
Carlitz~\cite{Carlitz78} proved that
\begin{equation*}\label{CarlitzGF}
H(q,z)=\sum_{n=0}^\infty \frac{z^n}{n!}\sum_{k=0}^nR(n+1,k)q^{n-k}=\left(\frac{1-q}{1+q}\right)\left(\frac{\sqrt{1-q^2}+\sin (z\sqrt{1-q^2})}{q- \cos (z\sqrt{1-q^2})}\right)^2.
\end{equation*}

Instead of considering $\sum_{k=0}^nR(n+1,k)q^{n-k}$ as Carlitz did, we shall study
$$
R(q,z)=\sum_{n\geqslant 0}R_{n+1}(q)\frac{z^n}{n!}.
$$
It is clear that $R(q,z)=H(\frac{1}{q},qz)$. Hence
$$
R(q,z)=\biggl(\frac{q-1}{q+1}\biggr)\biggl(\frac{\sqrt{q^2-1}+q\sin(z\sqrt{q^2-1})}{1-q\cos(z\sqrt{q^2-1})}\biggr)^2.
$$
Let $u=\sec^{-1}q$, i.e., $\sec u=q$. Then
\begin{equation*}
\begin{split}
\biggl(\frac{\sqrt{q^2-1}+q\sin(z\sqrt{q^2-1})}{1-q\cos(z\sqrt{q^2-1})}\biggr)^2&=\biggl(\frac{\sin u+\sin(z\sqrt{q^2-1})}{\cos u-\cos(z\sqrt{q^2-1})}\biggr)^2\\
&=\frac{\cos^2\Bigl(\frac{u-z\sqrt{q^2-1}}{2}\Bigr)}{\sin^2\Bigl(\frac{u-z\sqrt{q^2-1}}{2}\Bigr)}\\
&=\frac{1+\cos(u-z\sqrt{q^2-1})}{1-\cos(u-z\sqrt{q^2-1})}\\
&=\frac{{q}
+\cos(z\sqrt{q^2-1})+\sqrt{q^2-1}\sin(z\sqrt{q^2-1})}
{{q}-\cos(z\sqrt{q^2-1})-\sqrt{q^2-1}\sin(z\sqrt{q^2-1})},
\end{split}
\end{equation*}
where in the second, third and fourth equality, we have applied the following results:
\begin{equation*}
\begin{split}
\sin u+\sin(z\sqrt{q^2-1})&=2\sin\biggl(\frac{u+z\sqrt{q^2-1}}{2}\biggr)\cos\biggl(\frac{u-z\sqrt{q^2-1}}{2}\biggr),\\
\cos u-\cos(z\sqrt{q^2-1})&=-2\sin\biggl(\frac{u+z\sqrt{q^2-1}}{2}\biggr)\sin\biggl(\frac{u-z\sqrt{q^2-1}}{2}\biggr),\\
\cos(u-z\sqrt{q^2-1})&=2\cos^2\biggl(\frac{u-z\sqrt{q^2-1}}{2}\biggr)-1=1-2\sin^2\biggl(\frac{u-z\sqrt{q^2-1}}{2}\biggr),\\
\cos(u-z\sqrt{q^2-1})&=\frac{\cos(z\sqrt{q^2-1})+\sqrt{q^2-1}\sin(z\sqrt{q^2-1})}{q},
\end{split}
\end{equation*}
respectively. Thus, we have
\begin{equation}\label{Rqz}
R(q,z)=\biggl(\frac{{q}-1}{{q}+1}\biggr)\biggl(\frac{{q}
+\cos(z\sqrt{q^2-1})+\sqrt{q^2-1}\sin(z\sqrt{q^2-1})}
{{q}-\cos(z\sqrt{q^2-1})-\sqrt{q^2-1}\sin(z\sqrt{q^2-1})}\biggr).
\end{equation}
Alternatively, by imitating the proof of Theorem~\ref{mthm:01}, it can be shown that $R=R(q,z)$ satisfies the following linear partial differential equation
$$
(1-zq^2)\frac{\partial R}{\partial z}+q(q^2-1)\frac{\partial R}{\partial q}=2qR,
$$
with initial condition $R(q,0)=R_1(q)=1$ and whose solution is \eqref{Rqz}.

Combining~\eqref{Rqz} and Corollary~\ref{pqxyz-vqxyz}, we get the following result.
\begin{theorem}
We have
$$P(q,x,y,z)=e^{xz(y-1)}R(\sqrt{q},z)^{\frac{x}{2\sqrt{q}}}.$$
\end{theorem}



\end{document}